\documentclass[12pt,reqno]{amsart}
\usepackage{color}
\usepackage{comment}
\usepackage{lineno}

\def\R{\mathbb{R}}

\def\calF{\mathcal{F}}

\def\e{\varepsilon}

\def\hv{\hat{v}}

\theoremstyle{plain}
	\newtheorem{theorem}{Theorem}[section]
	\newtheorem{lemma}[theorem]{Lemma}
	\newtheorem{corollary}[theorem]{Corollary}
	
	\newtheorem{proposition}[theorem]{Proposition}
	\newtheorem{remark}[theorem]{Remark}

	\newtheorem{example}[theorem]{Example}
\theoremstyle{plain}

\newcommand{\di}{\mathop{}\!\mathrm{d}}

\makeatletter
 \@addtoreset{equation}{section}
\makeatother

\begin{document}
\title[Borderline dimension $N=10$]{Bifurcation diagram\\ for super-exponential Gel'fand problems\\ in the borderline dimension $N=10$}

\author{Kenta Kumagai}
\thanks{ORCiD of KK is 0000-0001-6907-6765.}
\thanks{KK was supported by JSPS KAKENHI Grant Number 26KJ0083, as part of the JSPS Research Fellowship for Young Scientists.}
\address{Graduate School of Mathematical Sciences, The University of Tokyo,
 3-8-1 Komaba, Meguro-ku, Tokyo 153-8914, Japan
}
\email{kumagai-kenta@g.ecc.u-tokyo.ac.jp}

\author{Yasuhito Miyamoto}
\thanks{ORCiD of YM is 0000-0002-7766-1849}
\thanks{YM was supported by JSPS KAKENHI Grant Numbers 24K00530, 25H00591.}
\address{Graduate School of Mathematical Sciences, The University of Tokyo,
 3-8-1 Komaba, Meguro-ku, Tokyo 153-8914, Japan
}
\email{miyamoto@ms.u-tokyo.ac.jp}

\begin{abstract}
We study the global structure of bifurcation diagrams for semilinear elliptic Dirichlet problems with supercritical nonlinearities in the unit ball. 
In particular, we focus on the borderline dimension $N = 10$, where the qualitative behavior of the bifurcation diagram is not determined solely by the growth rate of the nonlinearity. 
We show that the bifurcation curve is monotone, yielding uniqueness of classical solutions, for a class of nonlinearities including $f(u) = \exp((u+1)^p)$ with $p > 1$ and iterated exponential functions. 
Our approach is based on the construction of suitable singular subsolutions that satisfy a Hardy-type stability condition, avoiding the need for explicit representations of singular solutions. 
As a consequence, we show that, in dimension $N = 10$, these nonlinearities exhibit the same qualitative bifurcation diagram as the classical Gel'fand problem. 
We also characterize the monotonicity of the bifurcation curve in terms of the existence of global-in-time unbounded solutions to the associated parabolic problem. 
\end{abstract}

\date{\today}
\maketitle


\section{\bf Introduction and main theorems}
Let $B\subset\R^N$, $N\ge 3$, be the unit ball.
We consider the bifurcation problem
\begin{equation}\label{S1E1}
\begin{cases}
\Delta u+\lambda f(u)=0 & \text{for}\ x\in B,\\
u=0 & \text{for}\ x\in \partial B,\\
u>0 & \text{for}\ x\in B,\\
\end{cases}
\end{equation}
where $\lambda>0$ is a bifurcation parameter and $f$ exhibits supercritical growth.
Throughout the present paper, we assume that $f\in C^2[0,\infty)$ satisfies
\begin{equation}\label{f}
f(u)>0,\ f'(u)>0,\ f''(u)>0\quad\text{for}\ u\ge 0.
\end{equation}
By the symmetry result of Gidas–Ni–Nirenberg \cite{GNN79}, all positive classical solutions are radial and decreasing.
The solution set forms a curve that can be parametrized by the $L^\infty$-norm, that is, $\lambda = \lambda(\alpha)$ with $\alpha = \left\|u\right\|_{L^\infty(B)}$, and that starts from $\lambda(0)=0$.
A central object in the analysis is the possible presence of a singular solution $u^*$, satisfying $u^*(r) \to \infty$ as $r \to 0$.
The asymptotic behavior of the solution curve $\lambda(\alpha)$ as $\alpha \to \infty$ leads to two fundamentally different types of bifurcation diagrams:
\begin{itemize}
\item {\bf Type I}: 
The curve $\lambda(\alpha)$ converges to a finite limit $\lambda^*$ while oscillating around $\lambda^*$, yielding infinitely many solutions at $\lambda = \lambda^*$.
A singular solution $u^*$ exists for $\lambda=\lambda^*$.
\item {\bf Type II}:
The curve $\lambda(\alpha)$ increases monotonically toward $\lambda^*$, yielding uniqueness of classical solutions for $0 < \lambda < \lambda^*$.
A singular solution $u^*$ exists for $\lambda=\lambda^*$.
\end{itemize}
The complete bifurcation diagrams were studied in \cite{G63,JL72} for two canonical nonlinearities $f(u)=e^u$, $(u+1)^p$.
In the case $f(u)=e^u$, the bifurcation diagram is of Type I for $2<N<10$, while it is of Type II for $N\ge 10$.
Gel'fand \cite{G63} established the Type I bifurcation for $N=3$, and Joseph-Lundgren \cite{JL72} extended the analysis to all $N\ge 3$.
In the case $f(u)=(u+1)^p$ two important exponents play a crucial role:
\begin{equation}\label{qS}
p_S=\frac{N+2}{N-2}\quad\text{for}\ N>2,\qquad
p_{JL}=
\begin{cases}
1+\frac{4}{N-4-2\sqrt{N-1}} & \text{for}\ N>10,\\
\infty & \text{for}\ 2<N\le 10.
\end{cases}
\end{equation}
The first one is the critical Sobolev exponent and the second one is called the Joseph-Lundgren exponent.
In the case $f(u)=(u+1)^p$, the bifurcation curve is of Type I for $p_S<p<p_{JL}$, while it is of Type II for $p\ge p_{JL}$.
In particular, if $2<N\le 10$, the bifurcation diagram is always of Type I for $p>p_S$.
In these two cases, special transformations reduce the problem to planar autonomous systems, and phase-plane analysis is effective.

We now turn to more general nonlinearities.
Considerable attention has been paid to the equation $\Delta u+\lambda u+u^p=0$, $p>p_S$ and the problem was studied in \cite{B89,BN87,DF07,GW11,MP91}.
Since this equation can be transformed into the equation $\Delta u+\lambda (u+u^p)=0$ by scaling, it can be reduced to Problem~\eqref{S1E1}.
For other bifurcation problems for supercritical nonlinearities, see \cite{D00,D08,D13,KW18}.
Since such special transformations are no longer available, a different analytical approach is required.
After those studies, the second author of the present paper started to classify all the possible bifurcation diagrams for the cases $f(u)=u^p+o(u^p)$, $p>p_S$,  ($u\to\infty$) and $f(u)=e^u+o(e^u)$ ($u\to\infty$) in \cite{M14,M15}.
Combining results obtained in \cite{M14,M15,MN18,MN20}, we summarize the known results as follows:
\begin{itemize}
\item $f(u)=e^u+o(e^u)$, the bifurcation diagram is of Type I for $2<N<10$.
In \cite{M15} examples of nonlinearities exhibiting Type II bifurcation were given if $N\ge 10$.
\item $f(u)=u^p+o(u^p)$, the bifurcation diagram is of Type I for $p_S<p<p_{JL}$.
In \cite{MN18} examples of nonlinearities exhibiting Type II bifurcation were given if $p\ge p_{JL}$ and $N>10$.
\end{itemize}
The uniqueness of the singular solution $(\lambda^*,u^*)$, as well as the convergence of $u$ to $u^*$ as $\alpha\to\infty$, are important, and these properties were established in \cite{MN20} for the above two classes.
The case $N=10$ appears to be critical (i.e., borderline).
Moreover, these results suggest that Type I bifurcation is governed primarily by the growth rate of $f$, whereas Type II behavior depends on more delicate global properties of $f$.
Indeed, in the case where $f(u)=(u+\e)+(u+\e)^p$, $p>p_{JL}$, the bifurcation diagram has finitely many turning points, and in particular, at least one.\\

A classification of bifurcation diagrams in the case of a broader class of nonlinearities was initiated in \cite{M18}.
We consider the following model nonlinearities, which arise naturally and are treated in \cite{GG20,KW18} as higher-order analogues of the classical Gel'fand problem:
\begin{equation}\label{NMC}
f(u)=\exp({(u+1)^p})\ (p>1),\qquad f(u)=\underbrace{\exp(\cdots \exp(u)\cdots)}_{\text{$n\ge 2$ times}}.
\end{equation}
We introduce the following $q$-exponent:
\begin{equation}\label{q}
q_f=\lim_{u\to\infty}\frac{f'(u)^2}{f(u)f''(u)}.
\end{equation}
It is known that $q_f\ge 1$, provided the limit exists.
See \cite{FI18,M18}.
We define the growth rate of $f$ by
$$
p_f=\lim_{u\to\infty}\frac{uf'(u)}{f(u)}.
$$
By L'Hospital's rule, we have
$$
\frac{1}{p_f}=\lim_{u\to\infty}\frac{f(u)/f'(u)}{u}
=\lim_{u\to\infty}\left(1-\frac{f(u)f''(u)}{f'(u)^2}\right)
=1-\frac{1}{q_f},
$$
and hence $1/p_f+1/q_f=1$.
Thus, $q_f$ is the conjugate exponent of the growth rate of $p_f$.
The $q$-exponent defined by \eqref{q} was introduced in \cite{DF10}.

It is easy to verify that $q_f=p/(p-1)$ if $f(u)=u^p$ with $p>1$, and that $q_f=1$ if $f(u)=e^u$.
For example, if $f(u)=u^p(\log(u + e))^{\gamma}$ with $p>1$ and $\gamma\in\R$, then $q_f=p/(p-1)$.
Hence, even if the principal term of $f$ is not $u^p$, the $q$-exponent is determined and $q_f=p/(p-1)$.
The case $q=1$ corresponds to a super-power case.
It was shown by \cite{FI18,M18} that the case $q=1$ includes two model cases \eqref{NMC}.
As stated above, the $q$-exponent satisfies $q_f\ge 1$.
The power case $q_f>1$ and super-power case $q_f=1$ can be treated in a unified way.
We define the conjugate exponents of $p_S$ and $p_{JL}$ as follows:
$$
q_S=\frac{N+2}{4},\qquad q_{JL}=\frac{N-2\sqrt{N-1}}{4}.
$$
Both $q_S$ and $q_{JL}$ admit formal definitions for all $N\ge 1$.
However, $q_S>1$ if $N>2$, and $q_{JL}>1$ if $N>10$.
Hence, these conditions on $N$ coincide with those of \eqref{qS}.

When \eqref{q} exists and $q_f<q_S$, the uniqueness of the radial singular solution and the convergence $\lambda(\alpha)\to\lambda^*$ were proved in \cite{MN23}.
Therefore, the number of the turning points is important for classification.
In \cite{M18} a sufficient condition, which can be determined solely by the $q$-exponent, for a broader class of nonlinearities was obtained.
\begin{proposition}\label{S1P2}
Assume that $f\in C^2[0,\infty)$ satisfies \eqref{f} and that the $q$-exponent defined by \eqref{q} exists.
If $q_{JL}<q_f<q_S$, then the bifurcation diagram of \eqref{S1E1} is of Type I.
\end{proposition}
If $q_f>1$, then the assumption $q_{JL}<q_f<q_S$ is equivalent to $p_S<p_f<p_{JL}$, and if $q_f=1$, then it is equivalent to $2<N<10$.
Therefore, Proposition~\ref{S1P2} recovers Type I bifurcation results for the cases $f(u)=e^u+o(e^u)$, $u^p+o(u^p)$, and reveals that Type I bifurcation occurs for the two nonlinearities \eqref{NMC} if $2<N<10$.\\

We now study Type II bifurcations for a wider class of nonlinearities.
We describe the connection with singular solutions.
Brezis-Vazquez~\cite{BV97} provided a sharp characterization of Type II bifurcation.
Specifically, they showed that Type II bifurcation is equivalent to the stability of the singular solution, thereby revealing that the singular solution plays a central role in the overall structure of solutions.
Although their result applies to a general setting, we state the following version adapted to our setting.
\begin{proposition}\label{S1P1}
Let $N\ge 3$, and let $f\in C^2[0,\infty)$ be a function such that \eqref{f} holds.
Assume that a singular solution $(\lambda^*,u^*)$ of \eqref{S1E1} exists and that $u^*\in H^1_0(B)$.
Then, the bifurcation diagram is of Type~II if and only if
\begin{equation}\label{S1P1E1}
\int_B\left(|\nabla\varphi|^2-\lambda^*f'(u^*)\varphi^2\right)\di x\ge 0
\quad\text{for all } \varphi\in C_0^1(B).
\end{equation}
\end{proposition}
In \cite{BV97}, this proposition was proved for general bounded domains. However, even when the domain is $B$, it is generally difficult to determine the exact value of $\lambda^*$ and to obtain an explicit expression for the singular solution $u^*$.
In \cite{BV97}, the authors applied Proposition~\ref{S1P1} to the two cases $f(u)=e^u$ and $(u+1)^p$, thereby providing a proof of Type II bifurcation that differs from that in \cite{JL72}, since in these cases the exact values and explicit singular solutions are available. 

Rather than relying on Proposition~\ref{S1P1}, it is useful to establish a sufficient condition on $f$ for the occurrence of Type II bifurcations, as such a condition can be readily applied to concrete examples.
The following result was established in \cite{MN24}:
\begin{proposition}\label{S1P3}
Let $N\ge 10$.
Assume that $f\in C^2[0,\infty)$ satisfies \eqref{f} and
that the $q$-exponent defined by \eqref{q} exists.
If $q_f\le q_{JL}$ and there exist $q_0$ and $q_1$ such that
\begin{equation}\label{S1P3E1}
0\le q_0\le \frac{f'(u)^2}{f(u)f''(u)}\le q_1\quad \text{for}\ u\ge 0,
\end{equation}
\begin{equation}\label{S1P3E2}
q_1(2N-4q_0)\le\frac{(N-2)^2}{4},
\end{equation}
then the bifurcation diagram of \eqref{S1E1} is of Type II.
\end{proposition}
We consider the two nonlinearities \eqref{NMC}.
As mentioned above, in both cases, $q_f=1$, and hence the bifurcation diagram is of Type I for $2<N<10$ (Proposition~\ref{S1P2}).
On the other hand, it was proved in \cite{MN24} that, in both cases,
$$
\frac{7}{16}\le\frac{f'(u)^2}{f(u)f''(u)}\le 1\quad\text{for}\ u\ge 0.
$$
Consequently, 
\begin{equation}\label{11db}
1\cdot\left(2N-4\cdot\frac{7}{16}\right)\le\frac{(N-2)^2}{4}\iff N\ge 11,
\end{equation}
and hence both bifurcation diagrams are of Type II for $N\ge 11$.
The $10$-dimensional case remains unresolved within this framework.
When $N=10$, condition \eqref{S1P3E2} holds only when $q_0=q_1=1$, since $q_0\le 1\le q_1$.
In this case, it follows from \eqref{S1P3E1} that $\frac{f'(u)^2}{f(u)f''(u)} \equiv 1$.
Solving this differential equation, we obtain $f(u)=C_0e^{C_1u}$, which shows that Proposition~\ref{S1P3} does not apply to general nonlinearities.

The $10$-dimensional case is borderline and quite delicate.
Indeed, Type II bifurcation occurs for $f(u)=e^u$, whereas the first author showed in \cite{K25} that Type I bifurcation occurs for $f(u)=e^u(u+1)^{\gamma}$ with $\gamma>1/12$.
Since the associated $q$-exponents are equal to $1$ in both cases, this demonstrates that, in dimension $10$, the bifurcation behavior cannot be characterized solely in terms of the growth rate ($q$-exponent).
For the two nonlinearities in \eqref{NMC}, it was shown in \cite{K25} that the bifurcation diagrams have only finitely many turning points.
However, establishing the occurrence of Type II bifurcation is nontrivial, as it depends not only on the asymptotic behavior of $f$ but also on its delicate global properties.
Consequently, the borderline case falls outside the scope of the existing classification framework and requires a more refined analysis.

The main objective of this paper is to resolve this borderline case and thereby complete the classification of bifurcation diagrams for a broad class of supercritical nonlinearities.
In particular, we establish that Type II bifurcation occurs for several natural model nonlinearities that were not covered by the existing theory.
We now state our main result.
\begin{theorem}\label{T1}
Let $f(u)=\exp((u+1)^p)$ $(p>1)$ or $\exp(\cdots\exp(u)\cdots)$ ($n\ge 2$ times).
Then the bifurcation diagram of \eqref{S1E1} is of Type II for $N=10$.
Therefore, bifurcation diagrams exhibit the same qualitative behavior as the classical Gel'fand problem $\Delta u+\lambda e^u=0$ for these nonlinearities in all dimensions $N\ge 3$.
\end{theorem}
We derive an abstract sufficient condition for Type II bifurcation in Lemma~\ref{S2S2L1}.
We show in Corollary~\ref{S3C7} that for $N\ge 10$, the nonlinearities of the form
$$
f(u)=e^{g(u)}
$$
undergo Type II bifurcation, where $g\in C^3[0,\infty)$ satisfies
\begin{equation}\label{g}
g'>0,\qquad g''>0,\qquad g'^2-g''\ge 0,\qquad 2g''^2-g'g'''>0
\quad\text{for}\ u\ge 0.
\end{equation}
Combining Proposition~\ref{S1P2}, we see that the bifurcation curve of \eqref{S1E1} with \eqref{g} has the same qualitative behaviors as the classical Gel'fand problem for $N\ge 3$.

It is shown in Section~4 that the two nonlinearities \eqref{NMC} satisfy this condition.
Further examples are given in Examples~\ref{S4EX2} and \ref{S4EX3} including
\begin{align*}
f(u)&=\exp(\cdots\exp((u+1)^p)\cdots)&\quad &(\textrm{$n\ge 1$ times},\ p>1),\\
f(u)&=\exp(\cdots\exp(e(u+1)^{u+1})\cdots)&\quad &(\textrm{$n\ge 0$ times}),\\
f(u)&=\exp(\cdots\exp((u+e)^{(u+e)^p})\cdots)&\quad &(\textrm{$n\ge 0$ times},\ p\ge 1).
\end{align*}

\bigskip
Type II bifurcation is also related to the asymptotic behavior of solutions to the parabolic problem
\begin{equation}\label{PP}
\begin{cases}
\partial_tu=\Delta u+\lambda f(u) & \text{for}\ x\in B,\ t>0,\\
u=0 & \text{for}\ x\in\partial B,\ t>0,\\
u(x,0)=u_0(x) & \text{for}\ x\in B.
\end{cases}
\end{equation}
It was shown in \cite{BCMR96} that \eqref{PP} admits a global-in-time solution if and only if \eqref{PP} admits a possibly weak stationary solution.
Using the methods developed in \cite{BCMR96,PV95}, we prove the following result for general nonlinearities:
\begin{theorem}\label{T2}
Let $N\ge 3$.
Suppose that \eqref{f} holds, that the limit \eqref{q} exists, and that $q_f<q_S$.
Then, the following are equivalent:
\begin{itemize}
\item[{(i)}] The bifurcation diagram of \eqref{S1E1} is of Type II.
\item[(ii)] Problem~\eqref{PP} with $\lambda=\lambda^*$ and $u_0=0$ admits a global-in-time classical solution whose $L^{\infty}$-norm diverges as $t\to\infty$ (blow-up in infinite time).
We call this solution a grow-up solution.
\end{itemize}
Therefore, combining this with Proposition~\ref{S1P1}, we see that (i), (ii) and the following (iii) are equivalent:
\begin{itemize}
\item[(iii)] The singular solution of \eqref{S1E1} is stable in the sense of \eqref{S1P1E1}.
\end{itemize}
\end{theorem}
It was shown in \cite{PV95} that Theorem~\ref{T2}~(ii) holds for $f(u)=e^u$ and $N\ge 10$, since the singular solution has the explicit expression $(\lambda^*,u^*)=(2N-4,-2\log|x|)$ and the stability of $u^*$ can be verified by Hardy's inequality.
The case $f(u)=u^p$ with nonhomogeneous Dirichlet boundary conditions was studied in \cite{DGLV98}.
Note that an explicit singular solution is available also in the problem of \cite{DGLV98}.
The existence of grow-up solutions can be reduced to that of stable singular solutions.
However, determining the stability of singular solutions is not straightforward, and grow-up solutions are therefore not known for a broad class of nonlinearities.
By Proposition~\ref{S1P3} and Theorem~\ref{T1}, we obtain the following:
\begin{corollary}
Let $f(u)=\exp((u+1)^p)$ $(p>1)$ or $\exp(\cdots\exp(u)\cdots)$ ($n\ge 2$ times).
If $N\ge 10$, then (i), (ii), and (iii) in Theorem~\ref{T2} hold.
\end{corollary}

\bigskip
We explain the technical details and difficulties.
We consider $v(r)=u(\frac{r}{\sqrt{\lambda}})$.
Then $v$ satisfies the equation
$$
v''+\frac{N-1}{r}v'+f(v)=0.
$$
Since there exists a limit \eqref{q} with $q_f<q_S$, this equation has a unique singular solution $v^*$ (Proposition~\ref{S2S1P1}) and $v^*$ has the following asymptotic behavior
\begin{equation}\label{v^*}
v^*(r)=F^{-1}\left[\frac{r^2}{2N-4q_f}e^{-x^*(t)}\right]\quad\text{as}\ r\to 0,
\end{equation}
where $t=-\log r$, $F^{-1}$ is the inverse function of
\begin{equation}\label{F}
F(u)=\int_u^{\infty}\frac{\di s}{f(s)},
\end{equation}
and $x^*$ satisfies
$$
x^{*\prime\prime}-(N+2-4q_f)x^{*\prime}+(2N-4q_f)(e^{x^*}-1)+(q_f-1)(x^{*\prime})^2-(q_f-f'(v^*)F(v^*))(x^{*\prime}+2)^2=0.
$$
When $N=10$, we see $q_{JL}=1$.
Type I bifurcation occurs when $1=q_{JL}<q_f<q_S$ (Proposition~\ref{S1P2}).
Thus, we consider the case $q_f=1$.
When $N=10$, the above equation becomes
$$
x^{*\prime\prime}-8x^{*\prime}+16(e^{x^*}-1)-(1-f'(v^*)F(v^*))(x^{*\prime}+2)^2=0.
$$
Since it is difficult to study $x^*$ directly, we take a new approach.
We use a sufficient condition given by Lemma~\ref{S2S2L1} which is a variant of \cite[Proposition~3.1]{MN24}.
Lemma~\ref{S2S2L1} asserts that if there exists a stable singular subsolution $v\in H^1$ in the sense that
\begin{equation}\label{ub}
v''+\frac{N-1}{r}v'+f(v)\ge 0,\qquad
f'(v)\le\frac{(N-2)^2}{4r^2},
\end{equation}
then the bifurcation diagram is of Type II.
Although Lemma~\ref{S2S2L1} is related to Proposition~\ref{S1P1}, it does not require an explicit representation of the singular solution $v^*$, and is therefore more practical for constructing examples.
The function $v$ satisfying \eqref{ub} is automatically an upper bound for the singular solution $v^*$ (see Remark~\ref{R2}).
Interestingly, although $v$ provides an upper bound for $v^*$, it is not a supersolution.
We seek a function $v$ in the form
\begin{equation}\label{S1E10}
v(r)=F^{-1}\left[\frac{r^2}{2N-4q}e^{-x(t)}\right],\quad t=-\log r
\end{equation}
with $N=10$ and $q=1$.
The expression \eqref{S1E10} comes from the unique singular solution $v^*$ whose asymptotic behavior is given by \eqref{v^*}.
Note that the transformation $v\mapsto x$ is a generalization of Emden's transformation.

The main difficulty lies in constructing a function $x(t)$ such that $v(r)$ is a stable subsolution.
When $N=10$ and $q=1$, this problem is equivalent to finding $x(t)$ satisfying
$$
x''-8x'+16(e^x-1)-(1-f'(v)F(v))(x'+2)^2\ge 0,\qquad f'(v)\le 16e^{2t}.
$$
It is difficult to construct $x(t)$ using explicit combinations of elementary functions.
In particular, while it is relatively easy to construct a solution as $v\to\infty$, it is much harder to extend it so that two inequalities remain valid near $v=0$.
To overcome this difficulty, we define $x(t)$ implicitly (see \eqref{S3E1}).
Our choice of $v(r)$ provides a sharp approximation to $v^*$, both as $v\to\infty$ and near $v=0$.
Importantly, the global behavior of $v$ can be analyzed, whereas that of $v^*$ is difficult to study, and \eqref{ub} is difficult to verify for $v^*$.
Although each step in the construction is elementary, calculations involved in our approach are highly nontrivial.
The method developed in this paper may be useful for obtaining upper bounds for a wide class of singular solutions.

\bigskip
In Section 2, we recall known results about the singular solution, abstract sufficient conditions, and a transformation.
Some of them are not proved in the literature, and hence we prove them.
In Section 3, we construct the function $x(t)$ under the condition \eqref{g}.
In Section 4, we show that examples including \eqref{NMC} satisfy the sufficient condition.
In Section 5, we prove Theorem~\ref{T2}.

\section{{\bf Preliminaries}}
The purpose of this section is to establish a general framework for proving Type II bifurcation via the construction of suitable upper solutions.
The results in this section are valid for general $N$, not only for $N=10$.

We consider radial solutions of \eqref{S1E1}, for which the problem becomes
\begin{equation}\label{S2E1}
\begin{cases}
u''+\frac{N-1}{r}u'+\lambda f(u)=0 & \text{for}\ 0<r<1,\\
u(1)=0,\\
u>0 & \text{for}\ 0<r<1.
\end{cases}
\end{equation}
We use a new unknown function $v(r)=u(\frac{r}{\sqrt{\lambda}})$.
Problem \eqref{S2E1} is equivalent to the following:
\begin{equation}\label{S2E2}
\begin{cases}
v''+\frac{N-1}{r}v'+f(v)=0 & \text{for}\ 0<r<\sqrt{\lambda},\\
v(\sqrt{\lambda})=0,\\
v>0 & \text{for}\ 0<r<\sqrt{\lambda}.
\end{cases}
\end{equation}
Let $v(r,\alpha)$ be the unique solution of
\begin{equation}\label{S2E3}
\begin{cases}
v''+\frac{N-1}{r}v'+f(v)=0\quad \text{for}\ r>0,\\
v(0)=\alpha\ \ \text{and}\ \ v'(0)=0,
\end{cases}
\end{equation}
where $\alpha>0$ is a parameter.
Let $r_0(\alpha)$ denote the first zero of $v(r,\alpha)$.
Then $v$ is a bounded solution of \eqref{S2E2} if and only if $\sqrt{\lambda}=r_0(\alpha)$.
Thus, the bifurcation curve is given by $\lambda=r_0(\alpha)^2$.
To prove Type II bifurcation, we show that $r_0(\alpha)$ is increasing in $\alpha$.

\subsection{Singular solution}
We construct a singular solution to the problem
\begin{equation}\label{S2S1E1}
v''+\frac{N-1}{r}v'+f(v)=0
\end{equation}
near the singular point $r=0$ under the condition that $f$ is supercritical ($q_f<q_S$).
We assume that $f\in C^2[0,\infty)$ satisfies \eqref{f} and that $F(0)<\infty$.\\
The singular solution is written in terms of the inverse function of $F$ defined by \eqref{F}, we observe the following:
$$
\text{$F(u)<\infty$ for $u\ge 0$, $F(u)$ is decreasing, $\lim_{u\to\infty}F(u)=0$,}
$$
$$
\text{$F^{-1}(z)$ is defined for $0<z\le F(0)$, and $\lim_{\ z\to 0^+}F^{-1}(z)=\infty$.}
$$
\begin{proposition}\label{S2S1P1}
Let $N\ge 3.$
Assume that $f\in C^2[0,\infty)$, $f(u)>0$ and $F(u)<\infty$ for large $u>0$.
Assume also that the limit \eqref{q} exists and that $q_f<q_S$.
Then, \eqref{S2S1E1} admits a unique singular solution $v^*(r)$ for $0<r<\delta$ with some $\delta>0$, and the regular solution $v(r,\alpha)$ satisfies
\begin{equation}\label{S2S1P1E0-}
v(r,\alpha)\to v^*(r)\quad\text{in}\ C^2_{loc}(0,\delta]\quad\text{as}\ \alpha\to\infty.
\end{equation}
Furthermore, the positive singular solution $v^*$ satisfies
\begin{equation}\label{S2S1P1E0}
-r^{N-1}v^{*\prime}(r)=\int_0^rf(v^*(s))s^{N-1}\di s
\end{equation}
for $0<r<\delta$ and
\begin{equation}\label{S2S1P1E1}
v^*(r)=F^{-1}\left[\frac{r^2}{2N-4q_f}(1+\theta(r))\right]
\quad\text{as}\ r\to 0^+,
\end{equation}
where $F$ is defined by \eqref{F}, $\theta(r)\to 0$,
\begin{equation}\label{S2S1P1E2}
\theta(r)\in C^1(0,\delta)\quad\text{and}\quad
r\theta'(r)\to 0\ \text{as}\ r\to 0.
\end{equation}
\end{proposition}
All statements in Proposition~\ref{S2S1P1} except \eqref{S2S1P1E2} are the same as \cite[Theorem~1.1 and Lemma~2.5]{MN23}.
A brief proof of \eqref{S2S1P1E2} can be found in \cite[Proposition~3.2]{HM26}.

Under the condition \eqref{f}, both \eqref{S2E1} and \eqref{S2E2} admit singular solutions for a certain $\lambda^*$.
\begin{lemma}\label{S2S1L1}
Assume all assumptions of Proposition~\ref{S2S1P1}.
Assume, in addition, that \eqref{f} holds.
Then, the following assertions hold:\\
(i) There exists a unique $\lambda^*>0$ such that \eqref{S2E2} admits a singular solution $v^*$.
Furthermore, $v^*$ is a unique singular solution of \eqref{S2E2} with $\lambda=\lambda^*$,
and the following holds:
\begin{equation}\label{S2S1L1E0}
v(r,\alpha)\to v^*(r)\quad\text{in}\ C^2_{loc}(0,\infty)\quad\text{as}\ \alpha\to\infty.
\end{equation}
(ii) \eqref{S2E1} has a unique singular solution $(\lambda^*,u^*)$. Furthermore, $u^*\in H^1_0(B)$
\end{lemma}
\begin{proof}
(i) Let $v^*(r)$ be the singular solution near $0$ given in Proposition~\ref{S2S1P1}.
We extend $v^*$ in the positive direction of $r$.
It follows from \eqref{S2S1P1E0} that $v^{*\prime}\le 0$, and hence
$$
-r^{N-1}v^{*\prime}(r)=\int_0^rf(v^*(s))s^{N -1}\di s\ge f(v^*(r))\int_0^rs^{N-1}\di s=\frac{f(v^*(r))r^N}{N}.
$$
Integrating $-\frac{v^{*\prime}}{f(v^*)}\ge\frac{r}{N}$ over $(0,r]$, we have $F(v^*)\ge\frac{r^2}{2N}$,
and hence
$v^*(r)\le F^{-1}\left[\frac{r^2}{2N}\right]$.
By the assumption, $F(0)<\infty$.
Let $r_1=\sqrt{2NF(0)}$. Then,
$$
v^*(r_1)\le F^{-1}\left[\frac{r_1^2}{2N}\right]=0,
$$
which indicates that $v^*$ has the first zero $r_0$.
Let $\lambda^*=r^2_0$.
Then \eqref{S2E2} admits a singular solution, and hence \eqref{S2E1} admits a singular solution $(\lambda^*,u^*)$.\\
The uniqueness of $\lambda^*$ follows from the uniqueness of $v^*$.
It is obvious that $v^*$ is a unique singular solution of \eqref{S2E2} with $\lambda=\lambda^*$.
The convergence \eqref{S2S1L1E0} follows from \eqref{S2S1P1E0-}, because of the continuous dependence of the solution of \eqref{S2S1E1} with respect to initial data $(v(\delta),v'(\delta))$.
All assertions of (i) are proved.\\
(ii) Let $u^*(r)=v^*(\sqrt{\lambda^*}r)$.
Then, $(\lambda^*,u^*)$ is a singular solution of \eqref{S2E1}.
The uniqueness of the singular solution $(\lambda^*,u^*)$ of \eqref{S1E1} also follows from the uniqueness of $v^*$.
We omit the details.

We show that $u^*\in H^1_0(B)$.
Let $B_{\rho}$ denote the ball centered at $O$ with radius $\rho$.
It is enough to show that $v^*\in H^1_0(B_{\sqrt{\lambda^*}})$, since $u^*(r)=v^*(\sqrt{\lambda^*}r)$.
Because $q_f<q_S$, there exists $q_1\in (1,q_S)$ and $M>0$ such that $f'(u)F(u)\le q_1$ for $u>M$. We have
$$
\frac{d}{du}\left(f(u)F(u)^{q_1}\right)=\left(f'(u)F(u)-q_1\right)F(u)^{q_1-1}
$$
for $u\ge M$. Then $f(u)F(u)^{q_1}$ is nonincreasing in $u\ge M$, and hence we obtain
$$
f(u)F(u)^{q_1}\le C\quad\text{for}\ u\ge M,
$$
where $C>0$. 
From \eqref{S2S1P1E1} we obtain
$$
f(v^*)\le CF[v^*(r)]^{-q_1}\le Cr^{-2q_1}
$$
for small $r>0$.
Hence, 
\begin{equation}\label{S2S1L1E1}
f(v^*(r))=O(r^{-2q_1})\quad\text{as}\ r\to 0^+.
\end{equation}
We obtain
$$
\int_0^rf(v^*(s))s^{N-1}\di s=O(r^{N-2q_1})\quad\text{as}\ r\to 0,
$$
where we used $N-1-2q_1>\frac{N}{2}-2>-1$ by $q_1<q_S$.
Because of \eqref{S2S1P1E0}, we obtain $v^{*\prime}(r)=O(r^{1-2q_1})$ as $r\to 0$.
This implies that $v^*(r)=O(r^{2-2q_1})$ as $r\to 0^+$.
For small $\delta>0$, 
$$
\int_{B_{\delta}}|v^*|^2\di x\le C\int_0^{\delta}r^{3+N-4q_1}\di r<\infty\quad\text{and}\quad
\int_{B_{\delta}}|\nabla v^*|^2\di x\le C\int_0^{\delta}r^{1+N-4q_1}\di r<\infty,
$$
since $3+N-4q_1>-1$ and $1+N-4q_1>-1$.
Since $v^*$ is of class $C^1$ outside a neighborhood of the origin, we have $v^*\in H^1_0(B_{\sqrt{\lambda^*}})$, and hence $u^*\in H^1_0(B)$.
The proof is complete.
\end{proof}

\subsection{Monotonicity of the first zero}
Let $R>0$.
We consider the equation
\begin{equation}\label{S2S2E2}
Z''+\frac{N-1}{r}Z'+Q(r)Z=0\quad\text{for}\ 0<r\le R,
\end{equation}
where $Q\in C(0,R]$ and $Q(r)>0$ for $0<r\le R$.

The following sufficient condition for Type II bifurcation was obtained in \cite[Proposition~3.1]{MN24}.
The proof is omitted.
\begin{proposition}\label{S2S2P1}
Let $\hv\in C^2(0,R]$ be a function such that $\hv(r)\to\infty$ and $|\hv'(r)|\to\infty$ as $r\to 0$ and that there exists $R>0$ such that $\hv(r)>0$ for $0<r<R$, $\hv(R)=0$ and
\begin{equation}\label{S2S2P1E1}
\hv''+\frac{N-1}{r}\hv'+f(\hv)\ge 0\quad\text{for}\ 0<r\le R.
\end{equation}
Assume, in addition, that \eqref{S2S2E2} has a positive solution $Z$ for $0<r\le R$, and that $Q$ and $Z$ satisfy
\begin{equation}\label{S2S2P1E3}
f'(\hv(r))\le Q(r)\quad\text{for}\ 0<r\le R,
\end{equation}
and
\begin{equation}\label{S2S2P1E2}
\lim_{r\to 0}r^{N-1}\hv(r)Z'(r)=0\quad\text{and}\quad
\lim_{r\to 0}r^{N-1}\hv'(r)Z(r)=0.
\end{equation}
Then the solution $v(r,\alpha)$ of \eqref{S2E3} satisfies the following (i) and (ii):\\
(i) For $\alpha>0$, $r_0(\alpha)<R$ and $v(r,\alpha)<\hv(r)$ for $0<r\le r_0(\alpha)$.\\
(ii) If $0<\alpha_1<\alpha_2$, then $v(r,\alpha_1)<v(r,\alpha_2)$ for $0\le r\le r_0(\alpha_1)$, and hence $r_0(\alpha_1)<r_0(\alpha_2)$.
\end{proposition}

The following lemma plays a key role in our analysis: it shows that the bifurcation diagram is of Type II provided that there exists a function $v$ satisfying all the assumptions of Lemma~\ref{S2S2L1}.
In Lemma~\ref{S2S2L1}, unlike Proposition~\ref{S2S1P1},
$q$ is an arbitrary parameter and is not necessarily equal to $q_f$.
\begin{lemma}\label{S2S2L1}
Let $N\ge 6$ and $q<\frac{N}{2}$.
Assume that there exist $q_0$ and $q_1$ such that $\frac{1}{2}<q_0<1<q_1<q_S$ and
$q_0\le f'(u)F(u)\le q_1$ for $u\ge M$ with some $M\ge 0$.
Let $v(r)$, $0<r\le R$, be a smooth positive function such that $v(R)=0$ and
\begin{equation}\label{S2S2L1E1}
v(r)=F^{-1}\left[\frac{r^2}{2N-4q}(1+\theta(r))\right],
\end{equation}
where $F$ is defined by \eqref{F}, $\theta(r)\in C^1(0,R)$, and as $r\to 0^+$,
\begin{equation}\label{S2S2L1E2}
\theta(r)\to 0\quad\text{and}\quad r\theta'(r)\to 0.
\end{equation}
If the following two inequalities hold:
\begin{equation}\label{S2S2L1E3}
v''+\frac{N-1}{r}v'+f(v)\ge 0\quad \text{for}\ 0<r\le R,
\end{equation}
and
\begin{equation}\label{S2S2L1E4}
f'(v)\le\frac{(N-2)^2}{4r^2}\quad \text{for}\ 0<r\le R,
\end{equation}
then $r_0(\alpha)$ is increasing in $\alpha$, and hence the bifurcation diagram of \eqref{S1E1} is of Type II.
\end{lemma}
In this paper we choose $q=1$ in the above lemma.

We begin with several remarks before proving the lemma.
\begin{remark}
The singular solution $v^*$ satisfies all assumptions of Lemma~\ref{S2S2L1} except for \eqref{S2S2L1E4}.
Therefore, if $v^*$ satisfies \eqref{S2S2L1E4}, then Type II bifurcation occurs.
Lemma~\ref{S2S2L1} relaxes this condition, as $v$ is not necessarily a singular solution in view of \eqref{S2S2L1E3}.
\end{remark}

\begin{remark}\label{R2}
Suppose that $f$ satisfies all assumptions of Lemmas~\ref{S2S1L1} and that there exists a function $v$ satisfying all assumptions of Lemma~\ref{S2S2L1}.
Then, $v$ is an upper bound for the singular solution $v^*$.
Indeed, in the proof of Lemma~\ref{S2S2L1} below we take $\hv=v$ and apply Proposition~\ref{S2S2P1}.
Combining Lemma~\ref{S2S1L1}~(i) and Proposition~\ref{S2S2P1}, we obtain
$$
\hv(r)\ge v(r,\alpha)\to v^*(r)\quad\text{in}\ C^2_{loc}(0,R]\quad \text{as}\ \alpha\to\infty.
$$
Assume, in addition, that $f''>0$.
Then, \eqref{S2S2L1E4} implies $f'(v^*)\le f'(\hv)\le \frac{(N-2)^2}{4r^2}$.
By Hardy's inequality, we have
$$
\int_{B_R}\left(|\nabla\varphi|^2-f'(v^*)\varphi^2\right)\di x\ge
\int_{B_R}\left(|\nabla\varphi|^2-\frac{(N-2)^2}{4r^2}\varphi^2\right)\di x\ge 0
$$
for all $\varphi\in C_0^1(B_R)$, and hence $v^*$ is stable ($(\lambda^*,u^*)$ is also stable).
Therefore, Type II bifurcation follows also from Proposition~\ref{S1P1}.
\end{remark}

\begin{remark}
Lemma~\ref{S2S2L1} is formally valid for $N\ge6$.
On the other hand, it is shown in \cite{M18} that the bifurcation diagram is of Type I under the assumptions of Proposition \ref{S2S1P1} for $2<N<10$, which implies that there exists no subsolution satisfying the stability condition \eqref{S2S2L1E4}.
In this case, $N\ge 10$ is necessary.
\end{remark}

\begin{proof}[Proof of Lemma~\ref{S2S2L1}]
We first show that as $r\to 0^+$, 
\begin{equation}\label{S2S2L1E10}
v(r)=O(r^{2-2q_1})\quad\text{and}\quad v'(r)=O(r^{1-2q_1}).
\end{equation}
This is a variant of \cite[Lemma 2.2]{MN24}.
We include a proof of \eqref{S2S2L1E10} for completeness.
We use the same argument as in the proof of \eqref{S2S1L1E1}.
Then we obtain
\begin{equation}\label{S2S1P2E2}
f(v(r))=O(r^{-2q_1})\quad\text{as}\ r\to 0.
\end{equation}
Differentiating $F(v)=\frac{r^2}{2N-4q}(1+\theta(r))$ with respect to $r$, we have
\begin{equation}\label{S2S2L1E9}
-\frac{v'}{f(v)}=\frac{r(2+2\theta+r\theta')}{2N-4q}.
\end{equation}
By \eqref{S2S2L1E2}, the RHS is positive for small $r>0$.
Thus, $v'(r)<0$ for $0<r<\delta$ with some $\delta>0$.
We show that $v$ satisfies
\begin{equation}\label{S2S1P2E3}
\liminf_{r\to 0}\left(-r^{N-1}v'(r)\right)=0.
\end{equation}
Assume by contradiction that $\liminf_{r\to 0^+}\left(-r^{N-1}v'(r)\right)=\ell\neq 0$.
Since $v'(r)<0$ for small $r>0$, we have $\ell>0$.
Then there exist $\e\in (0,\ell)$ and $r_1>0$ such that $-r^{N-1}v'(r)\ge \ell-\e$ for $0<r\le r_1$.
Then it follows that
$$
-\frac{v'(r)}{f(v(r))}\ge\frac{\ell-\e}{r^{N-1}f(v(r))}
\quad\text{for}\ 0<r\le r_1.
$$
Integrating the above on $[\rho,r]$, we obtain
$$
F[v(r)]\ge F[v(r)]-F[v(\rho)]\ge
(\ell-\e)\int_{\rho}^r\frac{\di s}{s^{N-1}f(v(s))}.
$$
From \eqref{S2S1P2E2} we see that $r^{N-1}f(v(r))=O(r^{N-1-2q_1})$ as $r\to 0$.
Since $N-1-2q_1>\frac{N}{2}-2\ge 1$ by $N\ge 6$, we have
$$
\lim_{\rho\to 0}\int_{\rho}^r\frac{\di s}{s^{N-1}f(v(s))}=\infty.
$$
This is a contradiction, and hence \eqref{S2S1P2E3} holds.
Then there exists a sequence $r_k\to 0$ satisfying $r_k^{N-1}v'(r_k)\to 0$ as $k\to\infty$.
By \eqref{S2S2L1E3} we have
\begin{equation}\label{S2S1P2E4}
-\left(r^{N-1}v'(r)\right)'\le r^{N-1}f(v(r)).
\end{equation}
Integrating both sides of \eqref{S2S1P2E4} on $[r_k,r]$, and letting $k\to\infty$, we obtain
$$
-r^{N-1}v'(r)\le\int_0^rs^{N-1}f(v(s))\di s
$$
for small $r>0$.
From \eqref{S2S1P2E2} we obtain
$$
\int_0^rs^{N-1}f(v(s))\di s=O(r^{N-2q_1})\quad\text{as}\ r\to 0,
$$
since $N-1-2q_1>\frac{N}{2}-2>-1$.
Hence we obtain $v'(r)=O(r^{1-2q_1})$ as $r\to 0$.
This implies that $v(r)=O(r^{2-2q_1})$ as $r\to 0$.
We obtain \eqref{S2S2L1E10}. 

Next, we check all conditions of Proposition~\ref{S2S2P1}.
By \eqref{S2S2L1E1} and \eqref{S2S2L1E2} we see that $v(r)\to\infty$ as $r\to 0^+$.
Since $f'(u)F(u)\ge q_0$ for $u>M$, we have
$$
\frac{d}{du}\left(f(u)F(u)^{q_0}\right)=(f'(u)F(u)-q_0)F(u)^{q_0-1}
$$
for $u>M$. Then $f(u)F(u)^{q_0}$ is nondecreasing in $u>M$, and hence
$$
f(u)F(u)^{q_0}\ge C\quad\text{for}\ u\ge M,
$$
where $C>0$.
By \eqref{S2S2L1E9} we have
$$
-v'(r)=\frac{2+2\theta+r\theta'}{2N-4q}rf(v)\ge C\frac{2+2\theta+r\theta'}{1+\theta}r^{1-2q_0},
$$
and hence $|v'(r)|\to\infty$ as $r\to 0$.
Let $Q(r)=\frac{(N-2)^2}{4r^2}$ and $Z(r)=r^{-\frac{N-2}{2}}$.
Then, $Q(r)>0$ and $Z(r)>0$ for $0<r<R$, and $Z$ satisfies \eqref{S2S2E2}.
The conditions \eqref{S2S2P1E1} and \eqref{S2S2P1E3} follow from \eqref{S2S2L1E3} and \eqref{S2S2L1E4}, respectively.
By \eqref{S2S2L1E10} we have
$$
\lim_{r\to 0^+}|r^{N-1}v(r)Z'(r)|\le \lim_{r\to 0^+} Cr^{\frac{3}{2}N-1-2q_1}=0,\quad
\lim_{r\to 0^+}|r^{N-1}v'(r)Z(r)|\le \lim_{r\to 0^+} Cr^{\frac{3}{2}N-1-2q_1}=0,
$$
since $\frac{3}{2}N-1-2q_1>N-2>0$.
Thus, \eqref{S2S2P1E2} holds.
Since all conditions are satisfied, Proposition~\ref{S2S2P1} is applicable.
Then the bifurcation diagram is of Type II.
\end{proof}

\subsection{Transformation}
Assume that \eqref{f} holds and $0<q<\frac{N}{2}$.
Motivated by \eqref{S2S2L1E1}, we introduce a new variable $x(t)$ as follows:
\begin{equation}\label{S2S3E0}
v(r)=F^{-1}\left[\frac{r^2}{2N-4q}e^{-x(t)}\right],\quad
t=-\log r.
\end{equation}
Here, $x(t)$ can be written as follows:
$$
x(t)=\log \frac{r^2}{(2N-4q)F[v(r)]}.
$$
Let $r_0$ be the first zero of $v(r)$.
Since $0<v<\infty$ for $0<r<r_0$, there exists $t_0\in\R$, which is $t_0=-\log r_0$, such that 
\begin{equation}\label{S2S3E2}
\frac{e^{-2t_0-x(t_0)}}{2N-4q}=F(0)
\end{equation}
and
\begin{equation}\label{S2S3E3}
0<\frac{e^{-2t-x(t)}}{2N-4q}<F(0)\quad\text{for}\ t>t_0.
\end{equation}

In Section~3 we construct a function $v$ of the form \eqref{S2S3E0} satisfying all assumptions of Lemma~\ref{S2S2L1}.
Therefore, it is useful to derive a differential inequality for $x$.
\begin{proposition}\label{S2S3P1}
Let $v(r)$ and $x(t)$ be related by \eqref{S2S3E0}.
Assume that \eqref{f} and \eqref{S2S3E3} hold and that $0<q<\frac{N}{2}$.
Then $v$ satisfies
$$
v''+\frac{N-1}{r}v'+f(v) \ge 0 \quad \text{for $r\le r_0$}
$$
if and only if
\begin{equation}\label{S2S3P1E0}
x''-(N+2-4q)x'+(2N-4q)(e^x-1)+(q-1)(x')^2-(q-f'(v)F(v))(x'+2)^2\ge 0
\end{equation}
for $t\ge t_0$.
In particular, when $N=10$ and $q=1$, \eqref{S2S3P1E0} reduces to
\begin{equation}\label{S2S3P1E01}
x''-8x'+16(e^x-1)-(1-f'(v)F(v))(x'+2)^2\ge 0.
\end{equation}
\end{proposition}
A proof of the corresponding equality is given in \cite[Lemma~2.2]{MN23}.
Since \eqref{S2S3P1E01} is crucial, we provide a proof here.
\begin{proof}
Differentiating $F(v(r))=\frac{e^{-x(t)-2t}}{2N-4q}$ twice, we obtain
\begin{equation}\label{S2S3P1E1}
-\frac{v'(r)}{f(v)}=\frac{d}{dt}\left(\frac{e^{-x(t)-2t}}{2N-4q}\right)\frac{dt}{dr}=\frac{x'(t)+2}{2N-4q}e^{-x(t)-t}
\end{equation}
and
\begin{equation}\label{S2S3P1E2}
-\frac{v''(r)}{f(v)}+\frac{f'(v)}{f(v)^2}v'(r)^2=-\frac{e^{-x(t)}}{2N-4q}\left(
x''(t)-3x'(t)-2-x'(t)^2\right)
\end{equation}
for $r\le r_0$.
We calculate the second term of the LHS of \eqref{S2S3P1E2}.
Using $F(v(r))=\frac{e^{-x(t)-2t}}{2N-4q}$ and \eqref{S2S3P1E1}, we have
$$
\frac{f'(v)}{f(v)^2}v'(r)^2=f'(v)\frac{(x'(t)+2)^2}{(2N-4q)^2}e^{-2x(t)-2t}
=\frac{e^{-x(t)}}{2N-4q}f'(v)F(v)(x'(t)+2)^2 \quad \text{for $r\le r_0$}.
$$
Then, by \eqref{S2S3P1E2} we obtain
$$
\frac{v''(r)}{f(v)}=\frac{e^{-x(t)}}{2N-4q}\left(x''(t)-3x'(t)-2-x'(t)^2+f'(v)F(v)(x'(t)+2)^2\right)\quad \text{for $r\le r_0$}.
$$
From \eqref{S2S3P1E1} and $\frac{1}{r}=e^t$, we have
$$
\frac{v'(r)}{rf(v)}=-\frac{e^{-x(t)}}{2N-4q}(x'(t)+2)\quad \text{for $r\le r_0$}.
$$
Thus, we obtain
\begin{align*}
&(2N-4q)e^{x(t)}\left(\frac{v''}{f(v)}+\frac{N-1}{r}\frac{v'}{f(v)}+1\right)\\
&=x''(t)-x'(t)^2-(N+2)x'(t)+f'(v)F(v)(x'(t)+2)^2+(2N-4q)e^{x(t)}-2N\quad \text{for $r\le r_0$}.
\end{align*}
Note that $q(x'(t)+2)^2=qx'(t)^2+4qx'(t)+4q$.
Then we obtain \eqref{S2S3P1E0} and \eqref{S2S3P1E01}.
\end{proof}

\section{{\bf Type II bifurcation}}
We construct a function $v$ that satisfies all conditions in Lemma~\ref{S2S2L1} provided that $f(u)=e^{g(u)}$ with $g$ satisfying \eqref{g}.
In this section we assume
$$
N=10
$$
and we set
$$
q=1.
$$
Then, $2N-4q=16$.

Let $\eta(y)=F^{-1}(e^{-y})$.
We define $h$ by
\begin{equation}\label{h}
h(y)=1-f'(\eta(y))F(\eta(y))\quad (y\ge -\log F(0)).
\end{equation}
Following the construction in Subsection 2.3, we define a function $v$ by
\begin{equation}\label{v}
v(r)=F^{-1}\left[\frac{r^2}{16}e^{-x(t)}\right],\qquad t=-\log r,
\end{equation}
where the function $x(t)$ is defined by the implicit relation
\begin{equation}\label{S3E1}
4(e^{x}-1)-h(x+2t+\log 16)=0.
\end{equation}
Then, $v(r)$ can also be written as
$$
v(r)=\eta(\xi(t)),
$$
where
$$
\xi(t)=x(t)+2t+\log 16.
$$

In the following lemma, we show that $x(t)$ is well-defined and study the domain of $x$ under the assumptions \eqref{S3L1E0} and \eqref{S3L1E0+}.
\begin{lemma}\label{S3L1}
Assume that the following conditions hold:
\begin{equation}\label{S3L1E0}
h(y)\to 0\quad\text{as}\ y\to \infty,
\end{equation}
\begin{equation}\label{S3L1E0+}
h'(y)<0\quad\text{for}\ y>-\log F(0).
\end{equation}
Then, there exist a function $x(t)$ satisfying \eqref{S3E1} and $t_0\in\R$ such that the following properties hold:\\
(1) $x(t_0)+2t_0+\log 16=-\log F(0)$,\\
(2) $x(t)$ is defined for $t\ge t_0$, and $x$ is of class $C^2$,\\
(3) $|x(t)|\le 1$ for large $t>0$,\\
(4) $\xi'(t)>0$ for $t\ge t_0$, and hence $\xi(t)$ is monotonically increasing in $t$.
\end{lemma}
\begin{proof}
Let
$$
\Phi(x,t)=4(e^x-1)-h(x+2t+\log 16).
$$
We show that there exists $x\in [-1,1]$ satisfying \eqref{S3E1} for all sufficiently large $t>0$.
As $t\to\infty$,
$$
\Phi(1,t)\to 4(e-1)>0,\qquad \Phi(-1,t)\to 4(e^{-1}-1)<0,
$$
where we have used $\lim_{t\to\infty}h(2t\pm 1+\log 16)=0$ by \eqref{S3L1E0}.
Thus, by the intermediate value theorem, 
there exists $x\in [-1,1]$ satisfying \eqref{S3E1} for all sufficiently large $t>0$.
The assertion (3) holds.
Henceforth, $x(t)$ denotes a solution of \eqref{S3E1}.
Since $\Phi$ is of class $C^2$, we have
$$
\Phi_x(x,t)=4e^x-h'(x+2t+\log 16)>0,
$$
where we have used $h'<0$ by \eqref{S3L1E0+}.
Note that $x(t)$ is unique in $x\in[-1,1]$ for large $t>0$ because $\Phi_x>0$.
By the implicit function theorem, $x$ is of class $C^2$.
Differentiating $\Phi(x(t),t)=0$ in $t$, we have
\begin{equation}\label{S3L1E1}
4e^xx'-h'(\xi)(x'+2)=0\quad\text{and hence}\ x'=-\frac{2}{\zeta},
\end{equation}
where
\begin{equation}\label{S3L1E2}
\zeta(t)=1-\frac{4e^{x(t)}}{h'(\xi(t))}.
\end{equation}
Since $h'<0$ and $\zeta>1$, we have
$$
-2<x'<0\quad\text{and}\quad \xi'>0.
$$
Thus, (4) holds.
We extend the domain of $x(t)$ from $+\infty$ in the negative direction, and let $(t_1,\infty)$ denote the maximal domain of $x$.
We show by contradiction that there exists $t_0\in\R$ such that $x(t_0)+2t_0+\log 16=-\log F(0)$.
Assume the contrary, i.e., such a $t_0$ does not exist.
Since $\Phi_x>0$ holds globally, the implicit function theorem guarantees that the solution can be extended as long as $\xi(t)>-\log F(0)$.
Therefore, there exists $t_1\in [-\infty,\infty)$ such that
\begin{equation}\label{S3L1E3}
x(t)+2t+\log 16>-\log F(0)
\end{equation}
for $t_1<t<\infty$.
Since $-2<x'<0$, a blow-up does not occur, and hence we can extend the domain of $x$ to $\R$, i.e., $t_1=-\infty$.
Hence, $x(t)\to\infty$ as $t\to -\infty$, because of \eqref{S3L1E3}.
This is a contradiction, since by \eqref{S3E1},
$$
1-f'(0)F(0)=h(-\log F(0))>h(\xi(t))=4(e^{x(t)}-1)\to\infty
\quad{as}\ t\to -\infty.
$$
Thus, there exists $t_0\in\R$ such that (1) and (2) hold.
We have proved all the assertions.
\end{proof}
We define
$$
R=e^{-t_0}.
$$
Then, $0<r\le R$ is equivalent to $t\ge t_0$, since $t=-\log r$.
Hence, $v(r)$ is defined for $0<r\le R$.
The function $v$ is decreasing for $0<r\le R$, since
$$
v'(r)=-f(v)\xi'(t)e^{-\xi(t)-t}<0.
$$
Note that \eqref{S2S3E2} and \eqref{S2S3E3} hold, which implies that $v(R)=0$.

In what follows, we consider the case
$$
f(u)=e^{g(u)},
$$
where $g\in C^3[0,\infty)$ satisfies \eqref{g}.
\begin{lemma}\label{S3L2}
Assume \eqref{g}.
Then, the following hold:\\
(i) $\lim_{u\to\infty}f'(u)F(u)=1$ and $\lim_{u\to\infty}\frac{f'(u)^2}{f(u)f''(u)}=1$.\\
(ii) $2f(u)f''(u)^2-f'(u)^2f''(u)-f(u)f'(u)f'''(u)>0$ for $u\ge 0$.\\
(iii) $F(u)-\frac{f'(u)}{f(u)f''(u)}>0$ for $u\ge 0$.\\
(iv) $h(y)\to 0$ as $y\to\infty$.\\
(v) $h'(y)<0$ for $y\ge -\log F(0)$.
\end{lemma}
\begin{proof}
(i) We first show that
\begin{equation}\label{S3L2E2}
\lim_{u\to\infty}\frac{g''(u)}{g'(u)^2}=0.
\end{equation}
Since $\frac{d}{du}\left(\frac{g''(u)}{g'(u)^2}\right)
=\frac{g'g'''-2g''^2}{g'^3}<0$, it follows that $\frac{g''}{g'^2}$ is decreasing in $u$.
Since $\frac{g''}{g'^2}>0$, the limit $\ell=\lim_{u\to\infty}\frac{g''(u)}{g'(u)^2}$ exists and $\ell\ge 0$.
Suppose that $\ell>0$. Integrating $\frac{g''}{g'^2}\ge\ell$ over $[u_0,u]$, we have
$$
\frac{1}{g'(u_0)}\ge -\frac{1}{g'(u)}+\frac{1}{g'(u_0)}\ge \ell(u-u_0)\to\infty\quad\text{as}\ u\to\infty,
$$
which is a contradiction.
We obtain \eqref{S3L2E2}.
By L'H\^{o}pital's rule we have
\begin{equation}\label{S3L2E3}
\lim_{u\to\infty}f'(u)F(u)=\lim_{u\to\infty}\frac{F(u)}{1/f'(u)}
=\lim_{u\to\infty}\frac{f'(u)^2}{f(u)f''(u)}=\lim_{u\to\infty}\frac{1}{1+g''(u)/g'(u)^2}=1.
\end{equation}
(ii) We have
$$
2ff''^2-f'^2f''-ff'f'''=(2g''^2-g'g''')e^{3g}>0.
$$
(iii) Let
$$
\calF(u)=F(u)-\frac{f'(u)}{f(u)f''(u)}.
$$
It suffices to show that
\begin{equation}\label{S3L2E4-}
\calF(u)>0\quad\text{for}\ u\ge 0.
\end{equation}
To this end, we prove that
$$
\calF(u)\to 0\quad\text{as}\ u\to \infty,\qquad
\frac{d}{du}\calF(u)<0\quad\text{for}\ u\ge 0.
$$
By \eqref{S3L2E3} we have
$$
\lim_{u\to\infty}\calF(u)=\lim_{u\to\infty}\left(f'(u)F(u)-\frac{f'(u)^2}{f(u)f''(u)}\right)\frac{1}{f'(u)}=0.
$$
By (ii) we have
$$
\frac{d}{du}\calF(u)=-\frac{2ff''^2-f'^2f''-ff'f'''}{f^2f''^2}<0
$$
for $u\ge 0$.
Thus, \eqref{S3L2E4-} holds.\\
(iv) We observe that
\begin{equation}\label{S3L2E1}
\lim_{y\to\infty}\eta(y)=\lim_{z\to 0^+}F^{-1}(z)=\infty.
\end{equation}
By \eqref{S3L2E3} and \eqref{S3L2E1} we have
$$
\lim_{y\to\infty}h(y)=\lim_{u\to\infty}\left(1-f'(u)F(u)\right)=0.
$$
(v) By (iii) we have
$$
h'(y)=-\left(F(\eta(y))-\frac{f'(\eta(y))}{f(\eta(y))f''(\eta(y))}\right)f(\eta(y))f''(\eta(y))e^{-y}<0.
$$
\end{proof}
Thus, \eqref{S3L1E0} and \eqref{S3L1E0+} are verified, and $x(t)$ is now well-defined by Lemma~\ref{S3L1}.\\

Hereafter, we verify that $v$ satisfies all assumptions of Lemma~\ref{S2S2L1}.
\begin{lemma}\label{S3L3}
Assume \eqref{g}.
Let $t=-\log r$ and
\begin{equation}\label{theta}
\theta(r)=e^{-x(t)}-1.
\end{equation}
The following hold:\\
(i) As $t\to\infty$, 
$$
x(t)\to 0\quad\text{and}\quad x'(t)\to 0.
$$
(ii) As $r\to 0$, 
$$
\theta(r)\to 0\quad\text{and}\quad r\theta'(r)\to 0.
$$
\end{lemma}
\begin{proof}
(i) By Lemma~\ref{S3L1}~(3), $|x(t)|\le 1$ for large $t>0$.
Thus,
$$
\lim_{t\to\infty}\xi(t)=\lim_{t\to\infty}\left(x(t)+2t+\log 16\right)=\infty.
$$
By \eqref{S3E1} we have
$$
4\left(e^{x(t)}-1\right)=h\left(x(t)+2t+\log 16\right)\to 0\quad\text{as}\ t\to\infty,
$$
which implies that $x(t)\to 0$ as $t\to\infty$.

By \eqref{S3L1E1} and \eqref{S3L1E2} we have
\begin{equation}\label{S3L3E1}
x'(t)=-\frac{2}{\zeta(t)},\quad \zeta(t)=1-\frac{4e^{x(t)}}{h'(\xi(t))}.
\end{equation}
Since $\eta(y)\to\infty$ as $y\to\infty$, it follows from $e^{-y}=F(\eta(y))$ and Lemma~\ref{S3L2}~(i) that
$$
\lim_{y\to\infty}h'(y)=\lim_{y\to\infty}\left(\frac{f'(\eta(y))^2}{f(\eta(y))f''(\eta(y))}-f'(\eta(y))F(\eta(y))\right)\frac{f(\eta(y))f''(\eta(y))}{f'(\eta(y))^2}f'(\eta(y))F(\eta(y))=0.
$$
Since $h'<0$, by \eqref{S3L3E1} it follows that $\zeta(t)\to\infty$ as $t\to\infty$.
Therefore, by \eqref{S3L3E1}, $x'(t)\to 0$ as $t\to\infty$.\\
(ii) By (i) we have
$$
\lim_{r\to 0}\theta(r)=\lim_{t\to\infty}\left(e^{-x(t)}-1\right)=0,\quad
\lim_{r\to 0}r\theta'(r)=\lim_{t\to\infty}x'(t)e^{-x(t)}=0.
$$
The proof is complete.
\end{proof}

Lemmas~\ref{S3L4-} and \ref{S3L4} are the main technical lemmas in this section.
\begin{lemma}\label{S3L4-}
Assume \eqref{g}. Then
\begin{equation}\label{S3L4-E0}
\frac{h''(y)}{h'(y)}\le 4\quad\text{for}\ y\ge -\log F(0).
\end{equation}
\end{lemma}
\begin{proof}
By Lemma~\ref{S3L2}~(iii) we have
\begin{equation}\label{S3L4-E1-}
F(u)-\frac{f'(u)}{f(u)f''(u)}>0\quad\text{for}\ u\ge 0.
\end{equation}
Since
$$
F(u)=\int_u^{\infty}g'(s)e^{-g(s)}\frac{1}{g'(s)}\di s
=\frac{e^{-g(u)}}{g'(u)}-\int_u^{\infty}e^{-g(s)}\frac{g''(s)}{g'(s)^2}\di s,
$$
we have
\begin{equation}\label{S3L4-E1}
f'(u)F(u)=1-g'e^{g}\int_u^{\infty}e^{-g}\frac{g''}{g'^2}\di s\le 1.
\end{equation}
Since 
$$
f=e^g>0,\quad f'=g'e^g>0,\quad f''=(g'^2+g'')e^g>0,
$$
by Lemma~\ref{S3L2}~(ii) we have
\begin{equation}\label{S3L4-E2}
\frac{f'''}{f''^2}-\frac{2}{f'}+\frac{f'}{ff''}\le 0.
\end{equation}
Differentiating $h(y)=1-f'(\eta(y))F(\eta(y))$ twice with respect to $y$, we have
\begin{align*}
h'&=\left(-f''F+\frac{f'}{f}\right)\eta'=\left(-f''F+\frac{f'}{f}\right)fF,\\
h''&=\left(-ff'''F^2-f'f''F^2+3f''F-\frac{f'}{f}\right)\eta'
=\left(-ff'''F^2-f'f''F^2+3f''F-\frac{f'}{f}\right)fF,
\end{align*}
where we have used $\eta'(y)=f(\eta(y))F(\eta(y))$.
After straightforward but tedious calculations, by \eqref{S3L4-E1-}, \eqref{S3L4-E1}, and \eqref{S3L4-E2} we have
\begin{align*}
\frac{h''}{h'}
&=\frac{f'F}{F-\frac{f'}{ff''}}\left(\frac{ff'''}{f'f''}F-\frac{2}{f'}+\frac{f'}{ff''}\right)+f'F-1\\
&=f'F\frac{ff'''}{f'f''}+\frac{f''F}{F-\frac{f'}{ff''}}\left(\frac{f'''}{f''^2}-\frac{2}{f'}+\frac{f'}{ff''}\right)+f'F-1\\
&\le f'F\frac{ff'''}{f'f''}.
\end{align*}
If $f'''\le 0$, then the last term is negative, and hence \eqref{S3L4-E0} holds.
If $f'''>0$, by \eqref{g} we have
$$
\frac{ff'''}{f'f''}=4+\frac{g'''-3g'^3-g'g''}{g'(g'^2+g'')}
<4+\frac{2\frac{g''^2}{g'}-3g'^3-g'g''}{g'(g'^2+g'')}\\
=1+2\frac{g''}{g'^2}\le 3
< 4.
$$
It follows from the above inequality and \eqref{S3L4-E1} that \eqref{S3L4-E0} holds.
The proof is complete.
\end{proof}

\begin{lemma}\label{S3L4}
Assume \eqref{g}.
Then,
$$
v''+\frac{N-1}{r}v'+f(v)\ge 0\quad\text{for}\ 0<r\le R.
$$
\end{lemma}
\begin{proof}
By Proposition~\ref{S2S3P1} we prove that \eqref{S2S3P1E01} holds.
Let
$$
\zeta(t)=1-\frac{4e^{x(t)}}{h'(\xi(t))}
$$
as defined by \eqref{S3L1E2}.
Since $h'<0$ by Lemma~\ref{S3L2}~(v), we see that
\begin{equation}\label{S3L4E0}
\zeta\ge 1.
\end{equation}
By \eqref{S3L1E1} we have
\begin{equation}\label{S3L4E1}
x'=-\frac{2}{\zeta}.
\end{equation}
Differentiating $\zeta$ with respect to $t$, we have
\begin{multline}\label{S3L4E2}
\zeta'=-\frac{4e^xx'}{h'(\xi)}+\frac{4e^x\xi'h''(\xi)}{h'(\xi)^2}
=\frac{4e^x}{h'(\xi)}\frac{2}{\zeta}+\frac{4e^x}{h'(\xi)}(x'+2)\frac{h''(\xi)}{h'(\xi)}\\
=-\frac{2(\zeta-1)}{\zeta}-\frac{2(\zeta-1)^2}{\zeta}\frac{h''(\xi)}{h'(\xi)},
\end{multline}
where we have used $\frac{4e^x}{h'(\xi)}=1-\zeta$, $\xi'=x'+2$, and $x'=-\frac{2}{\zeta}$ by \eqref{S3L4E1}.
Differentiating \eqref{S3L4E1} with respect to $t$, by \eqref{S3L4E2} we have
\begin{equation}\label{S3L4E3}
x''=\frac{2\zeta'}{\zeta^2}=-\frac{4(\zeta-1)}{\zeta^3}-\frac{4(\zeta-1)^2}{\zeta^3}\frac{h''(\xi)}{h'(\xi)}.
\end{equation}
Substituting \eqref{h}, \eqref{S3E1}, \eqref{S3L4E1}, and \eqref{S3L4E3} into the LHS of \eqref{S2S3P1E01}, we have
\begin{align*}
&x''-8x'+16(e^x-1)-(1-f'(v)F(v))(x'+2)^2\\
&\qquad =x''-8x'+4h(\xi)-h(\xi)(x'+2)^2\\
&\qquad =\frac{4}{\zeta^3}\left\{7\zeta-3+(\zeta-1)^2\left(4-\frac{h''(\xi)}{h'(\xi)}\right)\right\}
+\frac{4h(\xi)}{\zeta^2}(2\zeta-1)\ge 0,
\end{align*}
where we have used $v=\eta(\xi(t))$, $\frac{h''}{h'}\le 4$ by Lemma~\ref{S3L4-}, and $\zeta\ge 1$ by \eqref{S3L4E0}.
Thus \eqref{S2S3P1E01} holds.
\end{proof}

\begin{lemma}\label{S3L5}
Assume \eqref{g}.
Then,
$$
f'(v(r))\le\frac{16}{r^2}\quad\text{for}\ 0<r\le R.
$$
\end{lemma}
\begin{proof}
Because of Lemma~\ref{S3L2}, we have
$$
h(y)\ge 0\quad\text{for}\ y>-\log F(0),
$$
and hence
\begin{equation}\label{S3L5E1}
h(\xi(t))\ge 0\quad\text{for}\ t>t_0.
\end{equation}
By \eqref{h} we have
\begin{equation}\label{S3L5E2}
f'(\eta(\xi(t)))F(\eta(\xi(t)))=1-h(\xi(t)).
\end{equation}
By \eqref{S3E1} we have
\begin{equation}\label{S3L5E3}
e^{x(t)}=1+\frac{1}{4}h(\xi(t)).
\end{equation}
Using \eqref{v}, \eqref{S3L5E2}, and \eqref{S3L5E3}, we have
\begin{multline*}
f'(v)=\frac{f'(\eta(\xi(t)))F(\eta(\xi(t)))}{F(\eta(\xi(t)))}
=\frac{16}{r^2}(1-h(\xi(t)))e^{x(t)}\\
=\frac{16}{r^2}\left(1-\frac{3}{4}h(\xi(t))-\frac{1}{4}h(\xi(t))^2\right)\le\frac{16}{r^2},
\end{multline*}
where we have used $h\ge 0$ by \eqref{S3L5E1}.
\end{proof}

\begin{corollary}\label{S3C7}
Let $f(u)=e^{g(u)}$.
Assume that $g$ satisfies \eqref{g}.
If $N\ge 10$, the bifurcation diagram of \eqref{S1E1} is of Type II.
Combining Proposition~\ref{S1P2}, we see that the bifurcation diagram of \eqref{S1E1} has the same qualitative behaviors as the classical Gel'fand problem for $N\ge 3$.
If $N\ge 10$, the radial singular solution, which uniquely exists, is stable.
\end{corollary}
The case $N\ge 11$ is a special case of \cite{MN24}.
We reproduce the proof of \cite[Theorem 1.4]{MN24} for completeness.
\begin{proof}
We show that all the assumptions of Lemma~\ref{S2S2L1} are satisfied for $N=10$.
We take $q_0=\frac{3}{4}$ and $q_1=2$.
Then $\frac{1}{2}<q_0<1<q_1<q_S$.
Since $f'(u)F(u)\to 1$ by \eqref{S3L2E3}, we have $q_0\le f'(u)F(u)\le q_1$ for large $u>0$.
We define $v$ by \eqref{v}, and $\theta$ by \eqref{theta}.
It follows from Lemma~\ref{S3L3}~(ii) that \eqref{S2S2L1E2} holds.
The conditions \eqref{S2S2L1E3} and \eqref{S2S2L1E4} follow from Lemmas~\ref{S3L4} and \ref{S3L5}, respectively.
All assumptions of Proposition~\ref{S2S2P1} are now verified.
Therefore, Type II bifurcation follows from Lemma~\ref{S2S2L1} for $N=10$.

Next, we consider the case $N\ge 11$.
By \eqref{g} we have
$$
0\le \frac{g''(u)}{g'(u)^2}\le 1.
$$
Since $\frac{f'(u)^2}{f(u)f''(u)}=\frac{1}{1+\frac{g''(u)}{g'(u)^2}}$, we obtain
$$
\frac{7}{16}\le\frac{1}{2}\le\frac{f'(u)^2}{f(u)f''(u)}\le 1.
$$
Let $q_0=\frac{7}{16}$ and $q_1=1$.
By Proposition~\ref{S1P3} with \eqref{11db}, Type II bifurcation follows for $N\ge 11$.

By Remark~\ref{R2}, the singular solution is stable in both cases.
\end{proof}

\section{{\bf Examples}}
\begin{corollary}\label{S4C1}
Suppose that $\phi(u)\in C^3[0,\infty)$ satisfies \eqref{g} and that $\phi(0)\ge\log 2$.
Let
\begin{equation}\label{S4C1E1}
\begin{split}
&f(u)=\exp(g_n(u)),\\
&\qquad\qquad\text{where}\ g_1(u)=\phi(u)\ \text{and}\ g_n(u)=\exp(g_{n-1}(u))\ \text{for}\ n\ge 2.
\end{split}
\end{equation}
Then, the bifurcation diagram of \eqref{S1E1} is of Type II for $N\ge 10$.
\end{corollary}
\begin{proof}
We by induction see that $g_n(0)>\log 2$ for all $n\ge 1$, because $g_n(0)=\exp(g_{n-1}(0))\ge 2>\log 2$ for $n\ge 2$.
Next, we prove by induction that \eqref{g} holds for all $n\ge 1$.
Since $g_1=\phi$, \eqref{g} holds for $n=1$.
Assume that \eqref{g} holds for $n=k-1$.
Since
\begin{align*}
g'_k&=g'_{k-1}e^{g_{k-1}}>0,\\
g''_k&=(g'^2_{k-1}+g''_{k-1})e^{g_{k-1}}>0,\\
2g''^2_k-g'_kg'''_k&=\left\{g'^4_{k-1}+g'^2_{k-1}g''_{k-1}
+\left(2g''^2_{k-1}-g'_{k-1}g'''_{k-1}\right)\right\}e^{2g_{k-1}}>0,
\end{align*}
the three inequalities hold for $n=k$.
Using $-g''_{k-1}\ge -g'^2_{k-1}$, we have
$$
g'^2_k-g''_k=\left(g'^2_{k-1}e^{g_{k-1}}-g'^2_{k-1}-g''_{k-1}\right)e^{g_{k-1}}
\ge\left(g'^2_{k-1}e^{g_{k-1}(0)}-2g'^2_{k-1}\right)e^{g_{k-1}}
\ge 0,
$$
where we used $e^{g_{k-1}(0)}\ge 2$.
Thus, all inequalities in \eqref{g} hold for $n=k$.
By induction, \eqref{g} holds for all $n\ge 1$.
\end{proof}

\begin{proof}[Proof of Theorem~\ref{T1}]
Let $g(u)=(u+1)^p$ with $p>1$.
We verify that $g$ satisfies \eqref{g}.
A direct computation yields
\begin{align*}
g' &=p(u+1)^{p-1}>0,\\
g''&=p(p-1)(u+1)^{p-2}>0,\\
g'^2-g''&=p(u+1)^{p-2}\left\{p(u+1)^p-p+1\right\}\ge 0,\\
2g''^2-g'g'''&=p^3(p-1)(u+1)^{2p-4}>0.
\end{align*}
Thus, the conclusion of Corollary~\ref{S3C7} holds for $f(u)=\exp((u+1)^p)$ with $p>1$.

Next, let $\phi(u)=\exp(u)$ and let $f$ and $g_n$ be defined by \eqref{S4C1E1}.
We verify that $\phi$ satisfies all conditions of Corollary~\ref{S4C1}.
We have $\phi(0)=1>\log 2$.
By direct calculation, $\phi$ satisfies \eqref{g}.
Thus, the conclusion of Corollary~\ref{S3C7} holds for $f(u)=\exp(\cdots\exp(u)\cdots)$ ($n\ge 2$ times).
\end{proof}

\begin{example}\label{S4EX2}
Let $\phi(u)=(u+1)^p$ with $p>1$.
It was already shown in the proof of Theorem~\ref{T1} that $\phi$ satisfies \eqref{g}.
We have $\phi(0)=1\ge\log 2$.
Therefore, the conclusion of Corollary~\ref{S4C1} holds for $f(u)=\exp(g_n(u))$ ($n\ge 1$) defined by \eqref{S4C1E1}.
\end{example}

\begin{example}\label{S4EX3}
The following are examples of $\phi$ of Corollary~\ref{S4C1}:
$$
1+(u+1)\log(u+1),\qquad (u+e)^p\log(u+e)\ (p\ge 1).
$$
Therefore, the conclusion of Corollary~\ref{S4C1} holds for $f(u)=\exp(g_n(u))$ ($n\ge 1$) defined by \eqref{S4C1E1}.
The verification is straightforward and therefore omitted.
\end{example}

\section{{\bf Grow-up phenomena}}
\begin{proof}[Proof of Theorem~\ref{T2}]
It is known from \cite{BCMR96} that there exists a unique global-in-time solution $u$ of \eqref{PP} with $\lambda=\lambda^*$ and $u_0=0$.
If the singular solution $(\lambda^*,u^*)$ is unstable, then there exists a minimal solution $\underline{u}\ge 0$ with $\lambda=\lambda^*$. By the comparison principle, we have
\[
0 \le u \le \underline{u} \quad \text{in } B \times (0,\infty),
\]
which implies that the grow-up phenomenon does not occur.

It suffices to prove that grow-up occurs when the singular solution $(\lambda^*,u^*)$ is stable.
In this case, the bifurcation diagram is of Type II by Proposition~\ref{S1P1}. Moreover, by the comparison principle, we obtain
\[
0 \le u \le u^* \quad \text{in } B \times (0,\infty).
\]
Multiplying \eqref{PP} by $u_t$ and integrating over $B\times (\tau,t)$ with $0<\tau<t$, we obtain
\begin{equation}\label{T2PE1}
\int_{\tau}^{t}\int_{B}u_s^{2}\,\di x\di s
+ \frac{1}{2}\int_{B}\left(|\nabla u(x,t)|^2 - |\nabla u(x,\tau)|^2\right)\,\di x
= \lambda^{*}\int_{B}\bigl(\tilde{F}(u(x,t))-\tilde{F}(u(x,\tau))\bigr)\,\di x,
\end{equation}
where 
$$
\tilde{F}(u):=\int_{0}^{u}f(s)\,ds\le f(u)u.
$$
Since $0 \le u \le u^*$ and $f$ is increasing, we have
\begin{equation}\label{T2PE2}
\int_B f(u(x,t))u(x,t)\di x \le \int_B f(u^*)u^*\di x \le C_1.
\end{equation}
Combining \eqref{T2PE1} and \eqref{T2PE2}, we obtain
$$
\|u_t\|_{L^2(B\times (\tau,\infty))} + \sup_{t>\tau}\|\nabla u(\cdot,t)\|_{L^2(B)} \le C_2.
$$

We next show that $u(\,\cdot\,,t) \to u^*$ in $L^2(B)$ as $t\to\infty$.
Suppose the contrary.
Then there exist a sequence $t_n \to \infty$ and $\e>0$ such that
\begin{equation}\label{T2PE4}
\|u(\,\cdot\,,t_n)-u^*\|_{L^2(B)} > \varepsilon.
\end{equation}
Since $u_t \in L^2(B\times (\tau,\infty))$, by passing to a subsequence of $t_n$ if necessary, we may assume that
\[
\int_{t_n}^{t_n+1}\int_B |u_t|^2 \,\di x\di s \to 0.
\]
Hence, we can choose $\tau_n \in (t_n,t_n+1)$ such that
\[
\|u_t(\,\cdot\,,\tau_n)\|_{L^2(B)} \to 0.
\]
Moreover,
\begin{equation}\label{T2PE5}
\|u(\,\cdot\,,\tau_n)-u(\,\cdot\,,t_n)\|_{L^2(B)} \to 0.
\end{equation}
By \eqref{T2PE4} and \eqref{T2PE5} we obtain
\begin{equation}\label{T2PE6}
\|u(\,\cdot\,,\tau_n)-u^*\|_{L^2(B)} > \frac{\varepsilon}{2}
\end{equation}
for sufficiently large $n$.
By the uniform bound in $H^1_0(B)$, we may assume that
\[
u(\,\cdot\,,\tau_n) \rightharpoonup u_\infty \quad \text{weakly in } H^1_0(B),
\]
and strongly in $L^2(B)$.
Passing to the limit in the weak formulation of \eqref{PP} and using that $\left\|u_t(\,\cdot\,,\tau_n)\right\|_{L^2(B)} \to 0$,
we see that $u_\infty$ is a stationary solution of \eqref{S1E1}. Since the bifurcation diagram is of Type II, it follows that $u_\infty = u^*$, which contradicts \eqref{T2PE6}.
Thus, $u(\,\cdot\,,t)\to u^*$ in $L^2(B)$ as $t\to\infty$.
\end{proof}

\bigskip
\noindent{\bf Acknowledgements}\\
KK was supported by JSPS KAKENHI Grant Number 26KJ0083, as part of the JSPS Research Fellowship for Young Scientists.
YM was supported by JSPS KAKENHI Grant Numbers 24K00530, 25H00591.\\

\noindent{\bf Conflict of interest}\\
The authors have no relevant financial or non-financial interests to disclose.\\

\noindent{\bf Data availability}\\
Data sharing is not applicable to this article as no new data were created or analyzed in this study.\\

\noindent{\bf AI declaration}\\
During the preparation of this work, the authors used ChatGPT-5.6 Luna for assistance with English language editing and proofreading.
The authors reviewed and edited the output as needed and take full responsibility for the content of the published article.

\end{document}